\documentclass[a4paper,11pt]{amsart}
\usepackage{amsmath,amssymb,amsthm,amscd,xypic,enumerate}

\title{Weight structures cogenerated by weak cocompact objects}

\author{George Ciprian Modoi}
\address{Babe\c s--Bolyai University, Faculty of Mathematics and Computer Science \\
1, Mihail Kog\u alniceanu, 400084 Cluj--Napoca, Romania}
\email{cmodoi@math.ubbcluj.ro}

\thanks{}

\subjclass[2010]{18E30, 16D90, 55U35}
\keywords{weak compact object, weak cocompact object, t-structure, weight structure, Brown representability}

\date{\today}

\renewcommand{\iff}{if and only if }

\newcommand{\la}{\longrightarrow}

\newcommand{\N}{\mathbb{N}}
\newcommand{\Z}{\mathbb{Z}}
\newcommand{\Q}{\mathbb{Q}}

\DeclareMathOperator{\Hom}{Hom} 
 
\DeclareMathOperator{\Ker}{Ker}

\DeclareMathOperator{\colim}{\underrightarrow{\textrm{\rm colim}}}
\DeclareMathOperator{\ilim}{\underleftarrow{\lim}}
\DeclareMathOperator{\hocolim}{\underrightarrow{\textrm{hocolim}}}
\DeclareMathOperator{\holim}{\underleftarrow{\textrm{\rm holim}}}

\newcommand{\A}{\mathcal{A}}

\newcommand{\Y}{\mathcal{Y}}
\newcommand{\C}{\mathcal{C}}

\newcommand{\CS}{\mathcal{S}}
\newcommand{\T}{\mathcal{T}}

\newcommand{\X}{\mathcal{X}}

\newcommand{\G}{\mathcal{G}}

\newcommand{\ModR}{\hbox{\rm Mod-}R}
\newcommand{\ModA}{\hbox{\rm Mod-}A}

\newcommand{\FlatR}{\hbox{\rm Flat-}R}

\newcommand{\ProjR}{\hbox{\rm Proj-}R}

\newcommand{\Ab}{\mathcal{A}b}

\newcommand{\opp}{^\textit{o}}

\newcommand{\Add}[1]{\mathrm{Add}({#1})}
\newcommand{\Prod}[1]{\mathrm{Prod}({#1})}
\newcommand{\Prd}{\mathrm{Prod}}

\newcommand{\Htp}[1]{\mathbf{K}({#1})}

\theoremstyle{plain}
\newtheorem{thm}{Theorem}[section]
\newtheorem{lem}[thm]{Lemma}
\newtheorem{prop}[thm]{Proposition}

\theoremstyle{definition}

\newtheorem{constr}[thm]{Construction}

\theoremstyle{remark}
\newtheorem{rem}[thm]{Remark}
\newtheorem{expl}[thm]{Example}

\begin{document}

\begin{abstract}
We study t-structures generated by sets of objects which satisfy a condition weaker than the compactness. We also study weight structures cogenerated by sets of objects satisfying the dual 
condition. Under some appropriate hypothesis, it turns out that the weight structure is right adjacent to the t-structure. 
\end{abstract}

\maketitle

\section*{Introduction}
 
Denote by $\T$ a triangulated category, and by $\T\opp$ its dual. If necessary, in this Introduction we assume implicitly that $\T$ has (co)products. An object $G\in\T$ is compact if the covariant hom-functor 
$\T(G,-)$ preserves coproducts. Let $\G\subseteq\T$ be a set of compact objects. Then one can construct the smallest localizing subcategory containing $\G$. 
Recall that a {\em triangulated subcategory} of $\T$ is a subcategory which closed under extensions and shifts, and a {\em (co)localizing subcategory} is a triangulated subcategory which is also 
closed under arbitrary coproducts (products) in $\T$ (all subcategories we consider in this paper are full). 
The construction of the smallest localizing subcategory can be generalized in various ways. For example, we can consider the smallest {\em aisle} containing $\G$, that is the smallest subcategory which is closed under extensions, 
coproducts and only positive shifts (see for example \cite[Theorem 6.1]{OPS19}). This aisle is a component of (and determines) a t-structure, 
as this concept was defined in \cite{BBD82} (the definition is recalled in Section \ref{constr} below). 
In this paper we try to formalize and dualize this approach. Since duals of compact objects, that is cocompact ones, rarely exist in triangulated category, in \cite{OPS19} are defined $0$-cocompact 
objects, notion which is obtained by weakening the dual of the definition of compact objects above. Inspired by this approach, we define weak compact and weak cocompact objects in $\T$, 
and consider t-structures generated by a set of weak compact objects, respectively weight structures cogenerated by a set of weak cocompact objects. Note that weight structure are in some sense 
duals of t-structures; they were independently defined in \cite{B10} and \cite{P10}, in this last paper being called co-t-structures. 

The paper is organized as follows: In Section \ref{constr} we take an object $X\in\T$, a set of objects $\C\subseteq\T$ and 
we construct inductively two towers of objects, called a $\C$-cophantom respectively a $\C$-cocellular tower associated to $X$. We learned this construction from \cite{AB00}, but here we 
performed it in the dual setting, since this case seems to be less known. In this way we ``approximate'' $X$ with objects which are constructed inductively from objects in $\C$.
Theorems \ref{w-str} and \ref{t-str} are similar to \cite[Theorems 6.6 and 6.1]{OPS19}, but in addition they are  completed with results concerning Brown representability for 
$\T\opp$, respectively $\T$; these completions are based on the approximation of an object $X$ described above (compare with \cite[Theorem 1.1]{M16}).   
The main result from Section \ref{main} is Theorem \ref{adjacent}, where there are given conditions under which the the t-structure generated by a set of weak compact objects is left adjacent to the 
weight structure cogenerated by a set of weak cocompact objects. The key observation here is that if $G$ is any object, then 
the covariant hom-functors $\T(G,-)$ behaves well with respect to inverse towers and homotopy limits involved in the construction of $\C$-cophantom and $\C$-cocellular towers. 
In Section \ref{examples} we focus on finding examples of weak (co)compact objects, in order to apply the results of the previous ones. We show that both symmetric pairs 
in arbitrary triangulated categories defined in the spirit of Krause's \cite{K02}, and the set of complexes which generates the homotopy category of projective modules 
inside the homotopy category of flat modules considered by Neeman in \cite{N11} fulfill some of our hypotheses.

\section{Cophantom and cocellular towers}\label{constr}
In this paper $\T$ will denote always a triangulated category, whose suspension functor is denoted by $\Sigma$. 
If it is convenient, we write a triangle $X\to Y\to Z\to\Sigma X$, shortly as $X\to Y\to Z\stackrel{+}\to$. Let $\CS\subseteq\T$ be a subcategory of $\T$; its 
left and right orthogonal $^\perp\CS$ and $\CS^\perp$ are by definition the (full) subcategories of $\T$ having as objects those $X$ for which $\T(X,S)=0$, respectively 
$\T(S,X)=0$ for all $S\in\CS$. We denote by $\Sigma^{\geq0}\CS=\bigcup_{n\geq0}\Sigma^n\CS$, $\Sigma^{\leq0}\CS=\bigcup_{n\leq0}\Sigma^n\CS$, respectively $\Sigma^\Z\CS=\bigcup_{n\in\Z}\Sigma^n\CS$
the completion of $\CS$ under positive shifts, negative shifts respectively all shifts. More generally, if $K\subseteq\Z$ then we denote $\Sigma^K\CS=\bigcup_{k\in K}\Sigma^k\CS$. 
Note that if $\CS$ is a small set, the same is true for these completions. 
Among others, these notations allow us to use only the left and right orthogonal defined above (for example the subcategory consisting of all $X\in\T$ for which $\T(\Sigma^nX,S)=0$ for all $n\leq0$ and 
all $S\in\CS$ is equal to ${^\perp\Sigma^{\geq0}\CS}$, so we don't need to index the symbol $\perp$, how is the usage in other papers, e. g. \cite{OPS19}). It is also clear that  $\Sigma\Sigma^{\geq0}\CS\subseteq\Sigma^{\geq0}\CS$, 
$\Sigma^{\leq0}\CS\subseteq\Sigma\Sigma^{\leq0}\CS$ and $\Sigma\Sigma^\Z\CS=\Sigma^\Z\CS$. 
We say that $\CS$ (co)generates $\T$ if for an object $X\in\T$ the equality $\T(S,X)=0$ (respectively $\T(S,X)=0$) for all $S\in\Sigma^\Z\CS)$ implies $X=0$. Recall also that a {\em$\CS$-precover} of an object $X\in\T$ is a map $S_X\to X$ with $S_X\in\CS$, such that the induced map $\T(S,S_X)\to\T(S,X)$ is surjective for all 
$S\in\CS$. Dually we define an {\em $\CS$-preenvelope}. 

Recall that a {\em torsion pair} in $\T$ is a pair $(\X,\Y)$ of (full) subcategories, such that $\T(X,Y)=0$ for all $X\in\X$ and all $Y\in\Y$ and for all 
$T\in\T$ there is a triangle $X\to T\to Y\stackrel{+}\to$, with $X\in\X$ and $Y\in\Y$.  Further, this torsion pair is called a {\em t-structure}, respectively a {\em weight structure} 
(or {\em w-structure} for short), provided that $\Sigma\X\subseteq\X$ (or, equivalently, $\Y\subseteq\Sigma\Y$), respectively $\X\subseteq\Sigma\X$ (or, equivalently, $\Sigma\Y\subseteq\Y$).  
If $(\X,\Y)$ is a t-structure (w-structure), then the 
triangle $X\to T\to Y\stackrel{+}\to$, with $X\in\X$ and $Y\in\Y$ is called the {\em t-decomposition} (respectively a {\em w-decomposition}) of $T\in\T$. The difference between the definite, 
respectively undefinite article used before is explained by the fact that in the case of a t-structure this triangle is functorial in $T\in\T$, but is not at all unique 
in the case of a w-structure. Moreover if $(\X,\Y)$ is a t-structure, then the inclusion functors $\X\to\T$ and $\Y\to\T$ have a right, respectively left adjoint which are defined by  the assignments 
$T\mapsto X$ and $T\mapsto Y$, where  $X\to T\to Y\stackrel{+}\to$ is the t-decomposition of $\T$. A t-structure $(\X,\Y)$ is is called {\em stable} if $\Sigma\X=\X$ (or, equivalently, $\Y=\Sigma\Y$). 
In this case $\X$ is a {localizing subcategory}, and $\Y$ is a {colocalizing subcategory} of $\T$. Note that the notions of stable t-structure and stable w-structure coincide. 
Given two torsion pairs $(\X,\Y)$ and $(\mathcal{U},\mathcal{V})$ we say that $(\mathcal{U},\mathcal{V})$
{\em right adjacent} to $(\X,\Y)$ (and, symmetrically, $(\X,\Y)$ is {\em left adjacent} to $(\mathcal{U},\mathcal{V})$) if $\Y=\mathcal{U}$, case in which 
$(\X,\Y=\mathcal{U},\mathcal{V})$ is called a {\em TTF triple}. We call {\em (co)suspended} a TTF triples $(\X,\Y=\mathcal{U},\mathcal{V})$ for which $\Y\subseteq\Sigma\Y$ (respectively $\Sigma\Y\subseteq\Y$). If this is the case, then clearly $(\X,\Y)$ is a w-structure and $(\mathcal{U},\mathcal{V})$
is a t-structure (respectively, $(\X,\Y)$ is a t-structure and $(\mathcal{U},\mathcal{V})$ is a w-structure). Our interest in (co)suspended TTF triples originated in the fact that they appear in the theory of (co)silting objects (e.g. see \cite{A19} or \cite{MV18}).

Suppose now $\T$ has products, and let $\C\subseteq\T$ be a (small) set of objects. We denote by $\Prod\C$ the full
subcategory of $\T$ consisting of all direct factors of products of objects in $\C$.
We call {\em$\C$-cophantom} a map $X\to Y$ in $\T$ such that the induced map $\T(X,C)\leftarrow\T(Y,C)$ 
vanishes for all $C\in\C$. Observe that a map in $\T$ is $\C$-cophantom \iff it is $\Prod\C$-cophantom. Denote \[\Psi(\C)=\{\psi\mid\psi\hbox{ is a }\C\hbox{-cophantom}\},\] and observe that 
$\Psi(\C)$ is an ideal in $\T$, that is closed under sums (of  maps), and under compositions with arbitrary maps.

A diagram in $\T$ (indexed over $\N$) of the form \[X_0\stackrel{\xi_0}\leftarrow X_1\stackrel{\xi_1}\leftarrow X_2\leftarrow\cdots\] is called an {\it inverse tower}. 
A dual diagram is called a {\it direct tower}. Sometimes we skip the adjectives ``inverse'' or ``direct'' when they are clear from the context. 
An inverse tower is called {\it$\C$-cophantom tower} if all its connecting maps are $\C$-cophantoms.  Recall that the 
homotopy limit of the inverse tower above is defined (up to a non--canonical isomorphism) by the triangle 
\[ \holim X_i\la\prod_{i\in\N}X_i\stackrel{1-\xi}\la\prod_{i\in\N}X_i\stackrel{+}\la, \]
where the $j$-th component of $\xi$ is $\prod_{i\in\N}X_i\to X_{j+1}\stackrel{\xi_j}\to X_j$
(see \cite[dual of Definition 1.6.4]{N01}).

\begin{lem}\label{C-pre}
 If $\C$ is a set of objects, then every object $X\in\T$ has a $\Prod\C$-preenvelope.
\end{lem}

\begin{proof}
 We put $C_X=\prod_{C\in\C}C^{\T(C,X)}$ and the $\Prod\C$-preenvelope of $X$ is the unique map 
 $X\to C_X$ whose composition with the $\alpha$-th projection $C_X\to C$ is equal to $\alpha$, for all $\alpha\in\T(C,X)$ and all $C\in\C$. 
\end{proof}

\begin{constr}\label{cophantom}
 Let $\C$ be a set of objects as before, and let $X\in\T$. Associated to this data we will construct an inverse tower as following:
 Put ${Y}_0=X$, and if ${Y}_i$ with $i\geq0$ is given, then we construct inductively ${Y}_{i+1}$ via the triangle 
 \[{Y}_{i+1}\stackrel{\psi_i}\to {Y}_i\to C_i\stackrel{+}\to\]
 where ${Y}_i\to C_i$ is the $\Prod\C$-preenvelope of ${Y}_i$.
  Note that provided that ${Y}_i$ is given, even if we construct the preenvelope in a fixed manner as in Lemma \ref{C-pre}, the object ${Y}_{i+1}$ is determined only up to a non-canonical 
 isomorphism.  By the very definition of a preenvelope, we know that for all $C\in\C$ the second morphism in the following exact sequence 
 \[\T({Y}_{i+1},C)\leftarrow\T({Y}_i,C)\leftarrow\T(C_i,C)\]  is surjective, therefore $-\circ\psi_i:\T({Y}_i,C)\to\T({Y}_{i+1},C)$ vanishes.  Therefore all connecting maps 
 of the tower \[{Y}_0\stackrel{\psi_0}\leftarrow {Y}_1\stackrel{\psi_1}\leftarrow {Y}_2\leftarrow\cdots\] are $\C$-cophantoms, justifying the name we will give to it: 
 {\it a $\C$-cophantom tower} associated to $X$. Denote ${Y}_\infty=\holim {Y}_i$. 
\end{constr}

For a set of objects $\C\subseteq\T$ we define inductively $\Prd_0(\C)=\{0\}$ and $\Prd_{n+1}(\C)$ is
the full subcategory of $\T$ which consists of all objects $Y$
lying in a triangle $X\to Y\to Z\stackrel{+}\to$ with $X\in\Prod\C$
and $Z\in\Prd_n(\C)$. Clearly $\Prd_1(\C)=\Prod\C$ and 
the construction leads to an ascending chain of subcategories
$\{0\}=\Prd_0(\C)\subseteq\Prd_1(\C)\subseteq\Prd_2(\C)\subseteq\cdots.$ From
\cite[Remark 07]{N09} we learn that if $X\to Y\to Z\stackrel{+}\to$ is a
triangle with $X\in\Prd_n(\C)$ and $Z\in\Prd_m(\C)$ then
$Y\in\Prd_{n+m}(\C)$. Moreover if $\C$ is supposed to be closed under (de)suspensions, then the same is true for
$\Prd_n(\C)$, (see \cite[Remark 07]{N09} again).  
An object $X\in\T$ is called
{\em$\C$-cofiltered} if it can be written as a homotopy limit 
$X\cong\holim X_n$ of an inverse tower ${X}_0\leftarrow {X}_1\leftarrow {X}_2\leftarrow\cdots$, such that $X_0=0$ and the completion of 
${X}_{n+1}\to {X}_n$ leads to a triangle $C_{n}\to{X}_{n+1}\to{X}_n\stackrel{+}\to,$
with $C_n\in\Prod\C$. If this is the case, then it is easy to see inductively that
${X}_n\in\Prd_n(\C)$, for all $n\in\N$, and the inverse 
tower above is called a {\it$\C$-cofiltration} of $X$.

\begin{constr}\label{cocellular} 
 Fix a $\C$-cophantom tower associated to $X$, and use the notations from Construction \ref{cophantom}. Put $\psi^0=1_X:{Y}_0\to X$ and define inductively $\psi^{i+1}=\psi^i\psi_i:{Y}_{i+1}\to X$, 
 for all $i\geq0$.  
 Use the octahedral axiom in order to complete commutatively the diagram  
 \[\diagram &\Sigma^{-1}C_i\rdouble\ar@{.>}[d]&\Sigma^{-1}C_i\dto \\
        \Sigma^{-1}X\ddouble\rto&{Y}^{i+1}\rto\ar@{.>}[d]^{\delta_i}&{Y}_{i+1}\rto^{\psi^{i+1}}\dto^{\psi_i}&X\ddouble\\
        \Sigma^{-1}X\rto&{Y}^i\rto\ar@{.>}[d]&{Y}_i\rto_{\psi^i}\dto&X \\
        &C_i\rdouble&C_i&
\enddiagram\] with the dotted triangle in the second column. 
 Observe that ${Y}^0=0$ and the objects ${Y}^i$ for $i>0$ are determined again only up to a non-canonical isomorphism. Call 
 \[0={Y}^0\stackrel{\delta_0}\leftarrow {Y}^1\stackrel{\delta_1}\leftarrow {Y}^2\leftarrow\cdots\]
a {\it$\C$-cocellular tower} associated to $X$ (and to the initial $\C$-cophantom tower). Denote ${Y}^\infty=\holim {Y}^i$. The existence of the dotted triange in the diagram above 
implies that the $\C$-cocellular tower associated to $X$ is a $\C$-cofiltration of ${Y}^\infty$.
 \end{constr}

 We say that an object $C\in\T$ is {\it weak cocompact with respect to an inverse tower}
 \[X_0\stackrel{\xi_0}\leftarrow X_1\stackrel{\xi_1}\leftarrow X_2\leftarrow\cdots\] if the equalities $\T(\xi_i,C)=0$, for all 
 $i\geq0$, imply $\T(X_\infty,C)=0$, where $X_\infty=\holim X_i$. Remark that if $\T(\xi_i,C)=0$, for all 
 $i\geq0$, then obviously $\colim\T(X_i,C)=0$. 
 Note also that if $C$ is weak cocompact with respect to a tower as above, then for all $n\in\Z$, the object $\Sigma^nC$ is weak compact with respect to the tower shifted by $n$.
 
 \begin{thm}\label{w-str}
  Let $\C$ be a set of objects in a triangulated category with products $\T$, with the property every $C\in\C$ is weak cocompact with respect to any $\Sigma^{\geq0}\C$-cophantom tower. Then:
  \begin{enumerate}[{\rm (1)}]
   \item $\left({^\perp\Sigma^{\geq0}\C},({^\perp\Sigma^{\geq0}\C})^\perp\right)$ is a w-structure in $\T$. 
   \item $\left({^\perp\Sigma^\Z\C},({^\perp\Sigma^\Z\C})^\perp\right)$ is a stable w-structure (=stable t-structure) in $\T$ and $({^\perp\Sigma^\Z\C})^\perp$ is the smallest colocalizing subcategory of 
   $\T$ containing $\C$. 
   \item If $\T$ is the base of a strong stable derivator and $\C$ cogenerates $\T$, then $\T\opp$ satisfies Brown representability. 
  \end{enumerate}
 \end{thm}
 
 \begin{proof} For proving (1) let ${Y}_0\stackrel{\psi_0}\leftarrow {Y}_1\stackrel{\psi_1}\leftarrow {Y}_2\leftarrow\cdots$ be a $\Sigma^{\geq0}\C$-cophantom tower associated to an object $X\in\T$, 
  and let $Y_\infty=\holim Y_i$. Then its connecting maps are $\Sigma^{\geq0}\C$-cophantoms, and the same is true for all negative shifts of this tower. Therefore by hypotheses we get 
  $\T(\Sigma^nY_\infty,C)=0$, for all $n\leq0$ and all $C\in\C$. Next we apply the same argument as for (i) of \cite[Theorem 6.6]{OPS19}, the only modification being the use of the argument 
  above for showing that $Y_\infty\in{^\perp\Sigma^{\geq0}\C}$. For (2) we replace in (1) the set $\C$ by $\Sigma^{\leq0}\C$. Then obviously $\Sigma^{\geq0}(\Sigma^{\leq0}\C)=\Sigma^\Z\C$, 
  and it is routine to check that every object in $\Sigma^{\leq0}\C$ is weak cocompact with respect to any inverse tower whose connecting maps are $\Sigma^\Z\C$-cophantoms. The last statement of (2) 
  is proved in (ii) of \cite[Theorem 6.6]{OPS19}.
 
 For (3) note that in the proof of above cited result, the w-decomposition associated to an $X\in\T$ (with respect to the w-structure 
 whose existence is stated in (2) or, {\em mutatis mutandis}, in (1) too) is obtained by completion to a triangle of the map $Y_\infty\to X$. We do not recall precisely what a strong stable derivator is, since the unique feature of this notion we need is the result in \cite[Corollary 11.4]{KN13}: If $\T$ is the base of a strong stable derivator,
 then taking homotopy limits, we can choose the maps such that we get a triangle 
 \[{Y}_\infty\to X\to\Sigma{Y}^\infty\stackrel{+}\to\] therefore this is a w-decomposition of $X$. Then $\T({Y}_\infty,C)=0$ for all $C\in\Sigma^\Z\C$, and because $\C$ cogenerates $\T$, it follows that 
 ${Y}_\infty=0$. Thus the map $X\to\Sigma{Y}^\infty$ is an isomorphism.  Finally, by construction, ${Y}^\infty$ is $\C$ cofiltered, therefore every object in $\T$ is $\C$-cofiltered, 
 and we only have to apply \cite[Theorem 8]{M13}.
 \end{proof}

We say that the w-structure $({^\perp\Sigma^{\geq0}\C},({^\perp\Sigma^{\geq0}\C})^\perp)$ from Theorem \ref{w-str} is {\it cogenerated} by $\C$. In particular we can reformulate the 
fact that $\C$ cogenerates $\T$   
by saying the $(0,\T)$ is the stable t-structure cogenerated by $C$.

\begin{constr}\label{phantom}
We end this section with the remark that all constructions and considerations made in it can be dualized. Thus starting with a set of objects $\G$ in a triangulated category with coproducts $\T$,
we consider the ideal of $\G$-phantom maps
\[\Phi(\G)=\{\phi:X\to Y\mid\T(G,X)\stackrel{\phi\cdot-}\la\T(G,Y)\hbox{ vanishes for all }G\in\G\}.\] A direct tower whose connecting maps are $\G$-phantoms is called a {\it$\G$-phantom tower}. 

We can construct two (direct) towers associated to an object $X$ as follows: 
For $X\in\T$ we construct inductively a tower of the form \[Z_0\stackrel{\phi_0}\la Z_1\stackrel{\phi_1}\la Z_2\la\cdots,\]
by ${Z}_0=X$, and $G_i\to{Z}_i\stackrel{\phi_i}\to{Z}_{i+1}\stackrel{+}\to$ is obtained by the completion to a triangle of an $\Add\G$-precover of $Z_i$, for $i\geq0$. Here $\Add\G$ is 
the subcategory consisting of all direct summands of arbitrary direct sums of objects in $\G$. 
This is called {\em a $\G$-phantom tower associated to $X$}.
Further we construct a $\G$-cellular tower by completion 
of the following diagram with the triangle in the dotted line: \[\diagram &G_i\dto\rdouble&G_i\ar@{.>}[d]& \\
        X\ddouble\rto^{\phi^i}&{Z}_{i}\dto^{\phi_i}\rto&{Z}^{i}\rto\ar@{.>}[d]^{\gamma_i}&\Sigma X\ddouble\\
        X\rto_{\phi^{i+1}}&{Z}_{i+1}\rto\dto&{Z}^{i+1}\rto\ar@{.>}[d]&\Sigma X \\
        &\Sigma G_i\rdouble&\Sigma G_i&
\enddiagram\] Denote by $Z_\infty=\hocolim Z_i$ and $Z^\infty=\hocolim Z^i$. Then $Z^\infty$ is $\G$-filtered, (in the dual sense of above, when we said that $Y^\infty$ is $\C$-cofiltered). 
\end{constr}

An object $G\in\T$ is called {\em weak compact} with respect to a direct tower  \[{X}_0\stackrel{\xi_0}\to{X}_1\stackrel{\xi_1}\to{X}_2\to\cdots\] if $\T(G,\hocolim{X}_i)=0$, provided that 
all $\xi_i$ are $\{G\}$-phantoms. As one can expect, compactness implies weak compactness, 
since  if the object $G$ is compact, then for the tower above we have $\colim\T(G,{X}_i)\cong\T(G,\hocolim{X}_i)$. 
Finally we record the dual of Theorem \ref{w-str}:

\begin{thm}\label{t-str}
 Let $\G$ be a set of objects in a triangulated category with coproducts $\T$ with the property every $G\in\G$ is weak compact with respect to any $\Sigma^{\geq0}\G$-phantom tower. Then:
  \begin{enumerate}[{\rm (1)}]
   \item $\left(^\perp(\Sigma^{\geq0}\G^\perp),\Sigma^{\geq0}\G^\perp\right)$ is a t-structure in $\T$. 
   \item If $\left(^\perp(\Sigma^\Z\G^\perp),\Sigma^\Z\G^\perp\right)$ is a stable t-structure in $\T$ and $^\perp(\Sigma^\Z\G^\perp)$ is the smallest localizing subcategory of $\T$ containing $\G$. 
   \item If $\T$ is the base of a strong stable derivator and $\G$ generates $\T$, then $\T$ satisfies Brown representability. 
  \end{enumerate}
\end{thm}

The t-structure $(^\perp(\Sigma^{\geq0}\G^\perp),\Sigma^{\geq0}\G^\perp)$ from the previous Theorem is said to be {\it generated} by $\G$. For avoiding unnecessary brackets, 
we adopt the convention that $\Sigma$  is granted a higher precedence than $\perp$, that is $\Sigma^{\geq0}\G^\perp$ should be automatically read as $(\Sigma^{\geq0}\G)^\perp$.

\section{When the $\C$-cogenerated w-structure is right adjacent to the $\G$-generated t-structure}\label{main}

\begin{lem}\label{iso} Let $H:\T\to\A$ be a homological functor into an abelian AB4$^*$ category, anl let \[\diagram
\Sigma^{-1}X\rto\ddouble&Y^{i+1}\rto\dto_{\delta_i}&Y_{i+1}\rto\dto^{\psi_i}&X\ddouble\\
\Sigma^{-1}X\rto&Y^i\rto&Y_i\rto&X
\enddiagram\]
be a diagram in $\T$ whose rows are triangles. 
Denote by $Y^\infty=\holim Y^i$. 
\begin{enumerate}[{\rm (a)}]
 \item If $H(\psi_i)=0=H(\Sigma^{-1}\psi_i)=0$ for all $i\geq0$, then $\ilim H(Y^i)\cong H(\Sigma^{-1}X)$ and $\ilim^{(1)}H(Y^i)=0$. 
 \item If, in addition, $H(\Sigma^{-2}\psi_i)=0$ for all $i\geq0$ and $H$ preserves products, then the map $H(\Sigma^{-1}X)\to H(Y^\infty)$ is an isomorphism. 
\end{enumerate}
\end{lem}

\begin{proof} 
  (a). Applying the homological functor $H$ to the given diagram and using the hypothesis in (a) we get a commutative diagram with exact rows (consisting of solid arrows): 
{\footnotesize \[\diagram
H(\Sigma^{-1}{Y}_{i+1})\rto\dto^0&H(\Sigma^{-1}X)\rto\ddouble&H({Y}^{i+1})\rto\dto^{H(\delta_i)}\ar@{.>}[dl]&H({Y}_{i+1})\rto\dto^0&H(X)\ddouble\\
H(\Sigma^{-1}{Y}_i)\rto&H(\Sigma^{-1}X)\rto&H(Y^i)\rto&H(Y_i)\rto&H(X) 
\enddiagram\]}
From this diagram we deduce that both first and the last map in the first row are zero.  
Thus for $i\geq1$ the same is true for the first and the last map in the second row too. Finally, the factorization through the kernel produces the dotted arrow which makes the whole diagram commutative. 
Thus, starting with $i\geq2$, the inverse tower \[H({Y}^2)\leftarrow H({Y}^3)\leftarrow H({Y}^4)\leftarrow\cdots\] is the direct sum between the towers 
\[H({Y}_2)\stackrel{0}\leftarrow H({Y}_3)\stackrel{0}\leftarrow H({Y}_4)\leftarrow\cdots\hbox{
 and }H(\Sigma^{-1}X)\stackrel{=}\leftarrow H(\Sigma^{-1}X)\stackrel{=}\leftarrow\cdots\]
 and this proves the conclusion.
 
 (b). Since now $H(\Sigma^n\phi_i)=0$ for $n=0,-1,-2$, we can apply  (a), both to $H$ and to $H\cdot\Sigma^{-1}$. We get $\ilim H(Y^i)=H(\Sigma^{-1}X)$, 
 $\ilim H(\Sigma^{-1}Y^i)=H(\Sigma^{-2}X)$ and $\ilim^{(1)}H(Y^i)=0=\ilim^{(1)}H(\Sigma^{-1}Y^i)$. Since $\A$ is AB4$^*$ the respective inverse limit is defined by the exact sequence (see \cite[Remark A.3.6]{N01}):  
 \[0\to\ilim H(Y^i)\to \prod H(Y^i)\to \prod H(Y^i)\to 0,\]
 (where the map between products is ''1-shift'') and similar for $\ilim H(\Sigma^{-1}Y^i)$. 
 On the other hand, $H$ preserves products, therefore applying it to the triangle which define the homotopy limit $Y^\infty$ we get a commutative diagram 
 {\footnotesize\[\diagram \prod H(\Sigma^{-1}Y^i)\rto\dto^{\cong}&\prod H(\Sigma^{-1}Y^i)\rto\dto^{\cong}&0\ar@{.>}[d]&&\\
 H(\Sigma^{-1}\prod Y^i)\rto&H(\Sigma^{-1}\prod Y^i)\rto &H(Y^\infty)\ar@{.>}[d]\rto &H(\prod Y^i)\rto\dto^{\cong}&H(\prod Y^i)\dto^{\cong}\\
 &0\rto&\ilim H(Y^i)\rto\rto&\prod H(Y^i)\rto&\prod H(Y^i)
 \enddiagram \]} 
 which proves the conclusion.    
\end{proof}

\begin{rem}\label{diso} We can dualize Lemma \ref{iso} as follows: 
If $\T$ has coproducts, $H:\T\to\A$ is a homological coproduct preserving functor into an abelian AB4 category, then we $H\opp:\T\opp\to\A\opp$ preserves products, therefore if 
\[\diagram X\rto\ddouble&Z_{i}\rto\dto_{\phi_i}&Z^{i}\rto\dto^{\gamma_i}&\Sigma X\ddouble\\
X\rto&Z_{i+1}\rto&Z^{i+1}\rto&\Sigma X\enddiagram\] is a commutative diagram such that $H(\Sigma^{n}\phi_i)=0$ for all $i\geq0$ and all $n=-1,0,1$, then 
$H(X)\cong\colim H(\Sigma^{-1}Z^i)\cong H(\Sigma^{-1}Z^\infty)$, 
where $Z^\infty=\hocolim Z^i$. A similar conclusion can be drawn if we start with a cohomological (contravariant) functor $H:\T\to\A$, which sends coproducts into products ($\T$
has to have coproducts, and $\A$ has to be AB4$^*$): If $H(\Sigma^{n}\phi_i)=0$ for $n=-1,0,1$ and all $i\geq0$, then $H(X)\cong\ilim H(\Sigma^{-1}Z^i)\cong H(\Sigma^{-1}Z^\infty)$. 
But if we apply such a 
cohomological functor to the diagram considered in Lemma \ref{iso} above, then we get only $H(X)\cong\ilim H(\Sigma^{-1}Y^i)$ and (b) can not be more dualized.  
\end{rem}

\begin{thm}\label{adjacent} Let $\T$ be the base of a strong stable derivator and let $\G$ and $\C$ be two set of objects. 
\begin{enumerate}[{\rm (1)}]
 \item If $\T$ has products, all objects in $\C$ are weak cocompact with respect to any $\Sigma^{\geq0}\C$-cophantom tower and $\Psi(\C)^s\subseteq\Phi(\Sigma^{\{-1,0,1,2\}}\G)$ 
 for some $s\geq 1$, then ${^\perp\Sigma^{\geq0}\C}\subseteq\Sigma^{\geq0}\G^\perp$.
 \item If $\T$ has coproducts, all objects in $\G$ are weak compact with respect to any $\Sigma^{\geq0}\G$-phantom tower and 
 $\Phi(\G)^t\subseteq\Psi(\Sigma^{\{-2,-1,0,1\}}\C)$ for some $t\geq 1$, then $\Sigma^{\geq0}\G^\perp\subseteq{^\perp\Sigma^{\geq0}\C}$.
 \item If both of the hypotheses of (a) and (b) above are 
 fulfilled, then \[\left(^\perp(\Sigma^{\geq0}\G^\perp),\Sigma^{\geq0}\G^\perp={^\perp\Sigma^{\geq0}\C},({^\perp\Sigma^{\geq0}\C})^\perp\right)\] is a cosuspended TTF triple in $\T$.
 \item If $\T$ has products and coproducts, all objects in $\G$ are weak compact with respect to any 
 $\Sigma^{\geq0}\G$-phantom tower, $\Phi(\C)^s\subseteq\Phi(\Sigma^{\geq-1}\G)$ for some $s\geq1$ and $\C\subseteq(\Sigma^{\geq0}\G^\perp)^\perp$, 
 then $\C$ cogenerates a w-structure which is the right adjacent of the $\G$-generated t-structure, as in the case (3) above.
\end{enumerate}
\end{thm}

\begin{proof}
 (1). We claim first that if $Y_\infty$ is the homotopy colimit of a ${\Sigma^{\geq0}\C}$-cophantom tower associated to an arbitary object $X\in\T$, then $Y\in\Sigma^{\geq0}\G^\perp$. 
In order to show this, we start by pasting together $s$ successive diagrams obtained by Construction \ref{cophantom} in order to get the diagram 
\[\diagram
\Sigma^{-1}X\rto\ddouble&Y^{i+s}\rto\dto&Y_{i+s}\rto\dto^{\phi}&X\ddouble\\
\Sigma^{-1}X\rto&Y^i\rto&Y_i\rto&X
\enddiagram\]
Then $\phi\in\Psi(\C)^s\subseteq\Phi(\Sigma^{\{-1,0,1,2\}}\G)$. 
Then both functors \[\T(G,-),\T(\Sigma^{-1}G,-):\T\to\Ab,\] with $G\in\G$ arbitrary chosen, fulfill the hypothesis of Lemma \ref{iso} relative  to the above diagram.
Since the homotopy limit of a subtower is the same as those of the whole tower, see \cite[Lemma 1.7.1]{N01}, we get isomorphisms $\T(G,\Sigma^{-1}X)\cong\T(G,Y^\infty)$ and 
$\T(\Sigma^{-1}G,\Sigma^{-1}X)\cong\T(\Sigma^{-1}G,Y^\infty)$, therefore $\T(G,\Sigma^{-1}Y_\infty)=0$. Now the replacement of $G$ by $\Sigma^nG$ ($n\geq0$) in the argument above  
is the same as we would shift with $-n$ the diagram, hence we get the same conclusion $\T(\Sigma^nG, Y_\infty)=0$ for all $G\in\G$ and all $n\geq0$, which proves our claim. 

Next remark that the hypotheses on $\C$ assure us that 
$({^\perp\Sigma^{\geq0}\C},({^\perp\Sigma^{\geq0}\C})^\perp)$ is a w-structure on $\T$ (cf. Theorem \ref{w-str}), and for $X\in\T$ 
a w-decomposition is $Y_\infty\to X\to\Sigma Y^\infty\stackrel{+}\to$. Therefore if we start with $X\in{^\perp\Sigma^{\geq0}\C}$, then $X$ is a direct summand of 
$Y_\infty$ and, according to the claim above, $X\in\Sigma^{\geq0}\G^\perp$. 

(2).  This statement is the dual of (1).

(3). This statement follows immediately from (1) and (2).  

(4). Note that the hypothesis on $\G$ assures us that $(^\perp(\Sigma^{\geq0}\G^\perp),\Sigma^{\geq0}\G^\perp)$ is a t-structure (cf. Theorem \ref{t-str}). 
Moreover $(\Sigma^{\geq0}\G^\perp)^\perp$ is closed 
under positive shifts, therefore $\C\subseteq(\Sigma^{\geq0}\G^\perp)^\perp$ implies 
$\Sigma^{\geq0}\C\subseteq(\Sigma^{\geq0}\G^\perp)^\perp$, and we deduce $\Sigma^{\geq0}\G^\perp\subseteq{^\perp\Sigma^{\geq0}\C}$. If $Y_\infty$ is the homotopy colimit of a 
${\Sigma^{\geq0}\C}$-cophantom tower associated to an arbitrary object $X\in\T$, then we have showed in the proof of (1) that $Y_\infty\in\Sigma^{\geq0}\G^\perp$, thus
 $Y_\infty\in{^\perp\Sigma^{\geq0}\C}$. By the argument of \cite[Theorem 6.6]{OPS19} already used in Theorem \ref{w-str}, we learn that $({^\perp\Sigma^{\geq0}\C},({^\perp\Sigma^{\geq0}\C})^\perp)$ 
 is a w-structure.  We also know that the w-decomposition of  $X\in\T$ with respect to this w-structure is $Y_\infty\to X\to\Sigma Y^\infty\stackrel{+}\to$.  
 Thus if $X\in{^\perp\Sigma^{\geq0}\C}$ then $X\in\Sigma^{\geq0}\G^\perp$ as a direct summand of $Y_\infty$, which proves the inclusion ${^\perp\Sigma^{\geq0}\C}\subseteq\Sigma^{\geq0}\G^\perp$.
\end{proof}

\section{Weak compact and weak cocompact objects}\label{examples}

The aim of this section is to produce examples of sets of objects $\G$ and $\C$ which are weak compact, 
respectively weak cocompact with respect to appropriate direct, respective inverse towers. 
The first example in this sense is taken from \cite{OPS19} which is one of the most important source of inspiration for the present work.
Mittag-Leffler inverse towers are well--known and very useful in homological algebra. Dually a direct tower in an abelian category 
\[A_0\stackrel{\alpha_0}\la A_1\stackrel{\alpha_1}\la A_2\la\cdots\] is called {\em co-Mittag-Leffler} if for each $i$ the increasing sequence
\[0\subseteq\Ker(\alpha_i)\subseteq\Ker(\alpha_{i+1}\alpha_i)\subseteq\cdots\]
stabilizes. An object $C$ in a triangulated category with products $\T$  is called {\em $0$-cocompact}, 
provided that the equality $\T(\holim X_i,C)=0$ holds for each inverse tower $X_0\stackrel{\xi_0}\leftarrow X_1\stackrel{\xi_1}\leftarrow X_2\leftarrow\cdots$ with the properties that the direct tower 
\[\T(X_0,\Sigma C)\to\T(X_1,\Sigma C)\to\T(X_2,\Sigma C)\to\cdots\]
is co-Mittag-Leffler and $\colim\T(X_i,C)=0$. The proof of the following Lemma is implicitly contained in \cite[Proof of Theorem 6.6]{OPS19}; 
despite this fact we include the argument for the sake of completeness:

\begin{lem}\label{0-coco} 
 Let $\C$ be a set of $0$-cocompact objects, in a triangulated category with products $\T$. 
 Then objects in $\C$ are weak cocompact with respect to any $\Sigma^{\{0,1\}}\C$-cophantom (consequently to any $\Sigma^{\geq0}\C$-cophantom) tower.
\end{lem}

\begin{proof} Let $C\in\C$ and let  
$Y_0\stackrel{\psi_0}\leftarrow Y_1\stackrel{\psi_1}\leftarrow X_2\leftarrow\cdots$ be a tower such that $\psi_i\in\Psi(\Sigma^{\{0,1\}}\C)$. Put 
 $Y_\infty=\holim Y_i$.  Both  
 towers of abelian groups \[\T(Y_0,C)\to\T(Y_1,C)\to\T(Y_2,C)\to\cdots\] and \[\T(Y_0,\Sigma C)\to\T(Y_1,\Sigma C)\to\T(Y_2,\Sigma C)\to\cdots\]
 have zero connecting maps. Among others this implies for first one that $\colim\T(Y_i,C)=0$ and for the second one that it is co-Mittag-Leffler. Since 
 $C$ is $0$-cocompact we conclude $\T(C,Y_\infty)=0$.  
\end{proof}

\begin{expl}\cite[Theorem 6.8 and Corollary 6.10]{OPS19} Let $A$ be an algebra over a commutative artinian ring $k$, and let $E$ be the injective envelope of 
the regular left $A$-module $A$. Then for any finite complex $X$ of finitely generated left $A$-modules, the complex of right $A$-modules $\Hom_A(X,E)$ is $0$-cocompact in $\Htp\ModA$. In particular, 
if $A$ is an Artin algebra, then any finite complex of finitely generated $A$-modules is $0$-cocompact in $\Htp\ModA$.
\end{expl}

In what follow we want to study the weak compactness of $\kappa$-compact objects, with $\kappa$ being a regular cardinal (possible bigger than $\aleph_0$). 
Now $\T$ has to have coproducts. Recall from \cite{K02} that a set of objects $\G\in\T$ is called {\it perfect} provided that the induced map $\T(G,\coprod X_i)\to\T(G,\coprod Y_i)$, 
with $G\in\G$, is surjective for  any set of maps $\{X_i\to Y_i\mid i\in I\}$ in $\T$ with the property that $\T(G,X_i)\to\T(G,Y_i)$ are surjective for all $G\in\G$ and all $i\in I$.
Observe that if $\G\subseteq\T$ is perfect, then the same is true for $\Sigma^n\G$, for all $n\in\Z$.  
In what follows, for an additive category $\A$, we denote by $\widehat\A$ the category of finitely presented contravariant functors from $\A$ to $\Ab$, that is 
the category of functors $F:\A\opp\to\Ab$ admitting a presentation \[\A(-,Y)\to\A(-,X)\to F\to 0,\] for some $X,Y\in\A$.  

\begin{lem}\label{ab4} Let $\T$ be a triangulated category with coproducts. The set of objects $\G\subseteq\T$ is perfect \iff the restricted Yoneda functor 
\[H:\T\to\widehat{\Add\G},\ H(X)=\T(-,X)|_{\Add\G}\] preserves coproducts. 
Moreover, if this is the case then the category $\widehat{\Add\G}$ is abelian, AB4. 
\end{lem}

\begin{proof}
 The first equivalence is established in \cite[Lemma 3]{K02}. Further combining \cite[Lemma 2 and (the proof of) Lemma 3]{K02} we deduce that $\widehat{\Add\G}$ is the quotient category of 
 $\widehat\T$ modulo a coproduct closed Serre subcategory. Since products in $\widehat\T$ are exact (see for example \cite[Lemma 1]{K02}), the same remains true for its quotient $\widehat{\Add\G}$.  
\end{proof}

\begin{prop}\label{week-compact} Let $\T$ be a triangulated category with coproducts, and let $\G\subseteq\T$ be a perfect set of objects. 
Then objects in $\G$ are weak compact with respect to any $\Sigma^{\{-1,0\}}\G$-phantom tower. 
\end{prop}

\begin{proof} Let $Z_0\stackrel{\phi_0}\la Z_1\stackrel{\phi_1}\la Z_2\la\cdots$ be a direct tower in $\T$. 
Apply the restricted Yoneda functor $H:\T\to\widehat{\Add\G}$ to the diagram which defines $Z_\infty=\hocolim Z_i$. According to Lemma \ref{ab4}, this functor 
is coproduct preserving, thus we get an exact sequence of the form 
\[\coprod H(Z_i)\to\coprod H(Z_i)\to H(Z_\infty)\to\coprod H(\Sigma Z_i)\to\coprod H(\Sigma Z_i)\]
Now if we know in addition that  $\phi_i\in\Phi(\Sigma^{-1}\G\cup\G)$, for all $i\geq0$, then the first and the second map of this exact sequence are isomorphisms, 
showing that $H(Z_\infty)=0$, which means $\T(G,Z_\infty)=0$ for all $G\in\G$.
\end{proof}

\begin{rem}
In contrast with the case of $0$-cocompact objects, which fit in the hypotheses of Theorem \ref{w-str}, for a perfect set of objects $\G$ we do not know automatically that they are weak compact with respect to any $\Sigma^{\geq0}\G$-phantom tower, so we need an additional hypothesis in order to apply Theorem \ref{t-str} in order to deduce that $\G$ generates a 
t-structure. 
\end{rem}

A pair $(\G,\C)$ consisting of two sets of objects in a triangulated category with products and coproducts $\T$ is called a {\em symmetric pair} if $\Phi(\G)=\Psi(\C)$.
This definition is a reminiscent of the definition in \cite{K02} of a {\em symmetric set of generators}: A set of symmetric generators in $\T$ is a set of generators such that there is a set 
$\C\subseteq\T$, such that the pair $\Phi(\G)=\Psi(\C)$.

\begin{thm}\label{local} Let $\T$ be a triangulated category with products and coproducts, which is the base of a strong stable derivator, and let $(\G,\C)$ be a symmetric pair in $\T$. 
Then $\G$ is a perfect set and $\G^\perp={^\perp\C}$. Moreover: 
\begin{enumerate}[{\rm (1)}]
 \item If $\G$ are weak compact with respect to any $\Sigma^{\geq0}\G$-phantom tower then $\C$ cogenerates a w-structure and \[\left(^\perp(\Sigma^{\geq0}\G^\perp),\Sigma^{\geq0}\G^\perp={^\perp\Sigma^{\geq0}\C},({^\perp\Sigma^{\geq0}\C})^\perp\right)\] is a cosuspended TTF triple in $\T$. 
 \item If $\Sigma\G=\G$, then $\Sigma\C=\C$, and $\left(^\perp{(\G^\perp)},\G^\perp\right)$ and $\left({^\perp\C},({^\perp\C})^\perp\right)$ are stable t-structures in $\T$, the second one 
 being the right adjacent of the first. 
\end{enumerate}
\end{thm}

\begin{proof}
 Since $(\G,\C)$ is a symmetric pair in $\T$, then clearly $\Phi(\G)$ is closed under coproducts, what is equivalent by \cite[Proposition 2.1]{M10} to $\G$ being perfect. 
 Moreover $\G^\perp={^\perp\C}$ since for an object $X\in\T$ whose identity map is denoted 
 $1_X$ we have 
 \[X\in\G^\perp\Leftrightarrow 1_X\in\Phi(\G)\Leftrightarrow 1_X\in\Psi(\C)\Leftrightarrow X\in{^\perp\C}.\]
 For (1) and (2) we want to apply Theorem \ref{adjacent} (4). In order to verify its hypotheses, observe that the equality $\G^\perp={^\perp\C}$ implies $\C\subseteq(\Sigma^{\geq0}\G^\perp)^\perp$. 
 Note that in the hypotheses of (2), this inclusion is written simply $\C\subseteq(\G^\perp)^\perp$, and, because $\G$ is closed under all shifts, the same is true for $\Phi(\G)=\Psi(\C)$ and, then, 
 for $\C$. Finally objects of $\G$ are weak compact with respect to any $\Sigma^{\geq0}\G$-phantom tower in the case (1) by hypothesis, and in the case (2) by  Proposition \ref{week-compact}. 
 Now conclusion follows by Theorem \ref{adjacent}, with the mention that in the case (2) both the $\G$-generated t-structure and the $\C$-cogenerated w-structure are stable. 
\end{proof}

\begin{rem}\label{krause-brown}
In the particular case when $\T$ is the base of a strong stable derivator, the theorem above gives us \cite[Theorem A and Theorem B]{K02}: 
 If $\T$ has a symmetric set of generators, then both $\T$ and $\T\opp$ satisfy Brown representability theorem. Indeed, if $\G$ is a set of symmetric generators, then 
 $\G^\perp=0$, so Theorem \ref{local} implies that $T=^\perp{(\G^\perp)}=({^\perp\C})^\perp$. Thus every object of $\T$ is both $\G$-filtered and $\C$-cofiltered, and we apply 
 \cite[Theorem 8]{M14} and its dual. 
\end{rem}

\begin{expl}
 Let $\T$ be a triangulated category coproducts, which is the base of a strong stable derivator. Suppose also that $\T$ satisfies Brown representability; in particular $\T$ has to have products. 
 If $G\in\T$ is a compact object, then 
 the contravariant functor $\Hom_\Z(\T(G,-),\Q/\Z):\T\to\Ab$ is cohomological and sends coproducts into products. By Brown representability theorem it has to be  representable, and the object $C\in\T$ 
 (unique, up to a canonical isomorphism) which represents it is called the {\em Brown--Comenetz dual} of $G$.
 Clearly for every map $\phi$ in $\T$ we have $\T(G,\phi)=0$ \iff $\T(\phi,C)=0$, that is if $\G$ is a set (not necessary closed under sifts) of compact objects, and $\C$ is the set of 
 corresponding Brown--Comenetz duals of objects in $\G$, then $(\G,\C)$ is a symmetric pair in $\T$. Thus Theorem \ref{local} (1) tells us that 
 $(^\perp(\Sigma^{\geq0}\G^\perp),\Sigma^{\geq0}\G^\perp)$ is a t-structure whose right adjacent w-structure is $({^\perp\Sigma^{\geq0}\C},({^\perp\Sigma^{\geq0}\C})^\perp)$. 
 For a version of this result see \cite[Theorem 2.4.1]{B19}; there it is not required $\T$ to be the base of a strong stable derivator, but to satisfies both Brown representability and its dual instead.
\end{expl}

\begin{expl} Let $\T=\Htp\FlatR$ be the homotopy category of cochain complexes of flat (right) $R$-modules, where $R$ is a ring. Let $\G$ be the set of all bounded bellow complexes of 
finitely generated projective modules, up to homotopy. Then clearly $\Sigma\G=\G$. By \cite[Lemma 4.6]{N08} the set $\G$ is $\aleph_1$-perfect in the sense of Neeman, that is 
\cite[Definition 3.3.1]{N01}, therefore according to \cite[Lemma 4]{K01} it is also perfect in the sense of Krause's definition we already used before. Moreover $\G$ generates $\Htp\ProjR$, by 
\cite[Theorem 5.8]{N08}. On the other hand, by \cite[Theorem 3.2]{N10}, the inclusion 
$\Htp\FlatR\to\Htp\ModR$ has a right adjoint denoted by $J$. We call a {\em test complex} a bounded below complex $I$ of injective left $R$-modules, such that $H^i(I) = 0$ for all but 
finitely many $i\in\Z$, and for all $i\in\Z$, $H^i(I)$ must be isomorphic to a subquotient of a finitely generated, projective left $R$-module (here $H^i(I)$ is the $i$-th cohomology of 
the complex $I$). Denote \[\C=\{J\left(\Hom_\Z(I,\Q/Z)\right)\mid I\hbox{ runs over all test complexes}\}.\] 
 By \cite[Lemma 2.8]{N11} the ideal of $\C$ cophantom maps $\Psi(\C)$ consists exactly of so called {\em tensor phantom maps}, that is maps $f$ in $\Htp\FlatR$ such that $f\otimes_RI$ 
 induces zero map in all cohomologies, for any test complex $I$. Moreover, from \cite[Lemma 1.9]{N11}, we learn that the composition of two tensor phantom maps is a $\G$-phantom.  
 With our notations this means $\Psi(\C)^2\subseteq\Phi(\G)$.  Finally by \cite[Lemma 2.2 and Lemma 2.6]{N11}, we have the 
 inclusion $\C\subseteq(\G^\perp)^\perp$. Now by 
 Theorem \ref{adjacent} (4), we deduce that $\left(^\perp{(\G^\perp)},\G^\perp\right)$ and $\left({^\perp\C},({^\perp\C})^\perp\right)$ are stable t-structures in $\T$ ans the second 
 is adjacent to the first. More about this example can be found in \cite{M14} and \cite{M15}.  
\end{expl}

\end{document}